\documentclass[10pt]{amsart}
\usepackage[cp1251]{inputenc}
\usepackage[english,russian]{babel}
\usepackage{amsmath}
\usepackage{amssymb}
\usepackage{amsfonts}
\usepackage{graphicx}
\usepackage{tikz}
\usepackage{tikz-cd}
\usepackage{hyperref}
\usepackage{graphicx}
\usepackage{tikz} 
\usepackage{textcomp}
\usepackage{circuitikz}
\usetikzlibrary{math} 
\usepackage{caption}
\usepackage{subcaption}
\usepackage{pgfplots}
\pgfplotsset{compat=1.18}
\usepackage{pgfplotstable}
\usepackage{amssymb}

\DeclareMathOperator{\Int}{Int}

\usepackage{algorithm}
\usepackage{algpseudocode}
\usepackage{yfonts}
\usepackage{enumitem}

\newtheorem{lemma}{Lemma}
\newtheorem{theorem}{Theorem}

\newtheorem{proposition}{Proposition}

\theoremstyle{remark}
\newtheorem{definition}{Definition}
\newtheorem{remark}{Remark}

\begin{document}
	\renewcommand{\refname}{References}
	\renewcommand{\proofname}{Proof.}
	\renewcommand{\figurename}{Fig.}

	\thispagestyle{empty}
	
	\title[HOPF-TYPE THEOREMS FOR CONVEX POLYHEDRA]{HOPF-TYPE THEOREMS FOR CONVEX SURFACES}
	\author{Ilya M Shirokov}%
	\address{Ilya Mikhailovich Shirokov
		\newline\hphantom{iii} St. Petersburg Department
		\newline\hphantom{iii} of Steklov Mathematical Institute,
		\newline\hphantom{iii} Fontanka, 27,
		\newline\hphantom{iii} 191023, St. Petersburg, Russia}%
	\email{shirokov.im@phystech.edu}%
	
	\maketitle {\small
		\begin{quote}
			\noindent{\sc Abstract. } 
In this paper we study variations of the Hopf theorem concerning continuous maps $f$ of a compact Riemannian manifold $M$ of dimension $n$ to $\mathbb{R}^n$. We investigate the case when $M$ is a closed convex $n$-dimensional surface and prove that the Hopf theorem (as well its quantitative generalization) is still valid but with the replacement of geodesic to quasigeodesic in the sense of Alexandrov (and Petrunin). Besides, we study a discrete version of the Hopf theorem. We say that a pair of points $a$ and $b$ are $f$-neighbors if $f(a) = f(b)$. We prove that if $(P,d)$ is a triangulation of a convex polyhedron in $\mathbb{R}^3$, with a metric $d$, compatable with topology of $P$, and $f \colon P \to \mathbb{R}^2$ is a simplicial map of general position, then there exists a polygonal path in the space of $f$-neighbors that connects a pair of `antipodal' points with a pair of identical points. Finally, we prove that the set of $f$-neighbors realizing a given distance $\delta > 0$ (in a specific interval), has non-trivial first Steenrod homology with coefficients in $\mathbb{Z}$.

			\medskip
			
		\end{quote}
	}

\section{Introduction}

In this paper we continue the program of generalizations of the Hopf theorem, concerning continuous maps $f$ of a compact Riemannian manifold $M$ of dimension $n$ to Euclidean space $\mathbb{R}^n$ (see \cite{Previous} for a description of the program):

\begin{theorem}[H. Hopf \cite{Hopf}]
    Let $n$ be a positive integer, let $M$ be a compact Riemannian manifold of dimension $n$, and let $f\colon M \to \mathbb{R}^n$ be a continuous map. Then for any prescribed $\delta > 0$, there exists a pair $\{x,y\} \in M \times M$ such that $f(x) = f(y)$ and $x$ and $y$ are joined by a geodesic of length~$\delta$.
\end{theorem}

In \cite{Previous} we proved that the set of pairs of points satisfying the Hopf theorem when $\delta$ is fixed (such that there are no $\delta$-conjugate points) is uncountable (see \cite{Previous} for an exact formulation and proof of quantitative version of the Hopf theorem). However, the topological structure of the set of such pairs was not explored and remains unclear. Understanding this structure is one of the main driving forces for the present work.

The paper is organized as follows. First, in Sections 2 and 3, we extend the Hopf theorem (and its quantitative version) to the case when $M$ is a closed convex hypersurface (with geodesics replaced by quasigeodesics). Precise definitions of quasigeodesics are not necessary for the statements of our results and their proofs. However, definitions can be found in \cite{Alexandrov} (for two-dimensional convex surfaces) and in \cite{Petrunin}, \cite{Petrunin1} (for Alexandrov spaces with curvature $\ge \kappa$, and in particular, for convex hypersurfaces). See also \cite{Closed} for a beautiful exposition of quasigeodesics on two-dimensional convex polyhedra.

Next, in Section 4, we study a `discrete' version of the Hopf theorem concerning peace-wise linear maps of a boundary of a convex polyhedron $P \subset \mathbb{R}^3$ to $\mathbb{R}^2$ and prove several topological results on the structure of ``solutions'' of the Hopf theorem.

\section{Generalization of the Hopf theorem for convex surfaces}

To prove our first statement we need several auxiliary definitions and results. Then the statement follows almost immediately. 

\begin{definition}
    By a closed convex $n$-dimensional surface we mean the boundary of $n$-dimensional convex body in the Euclidean $(n+1)$-space $\mathbb{R}^{n+1}$. We assume that such a surface is provided with induced length metric from $\mathbb{R}^{n+1}$.
\end{definition}

\begin{definition}[Gromov-Hausdorff convergence \cite{Journey}]
    Let $\mathcal{X}_1, \mathcal{X}_2, ...$, and $\mathcal{X}_{\infty}$ be a sequence of complete metric spaces. Suppose that there is a metric $d$ on the disjoint union $$\mathbf{X} = \bigsqcup_{n \in \mathbb{N} \cup \{\infty\}} \mathcal{X}_n$$
    such that the restriction of the metric on each $X_n$ and $X_{\infty}$ coincides with its original metric, and $X_n \to X_{\infty}$ as subsets in X in the sense of Hausdorff. In this case we say that the metric on $\mathbf{X}$ defines convergence $\mathcal{X}_n \xrightarrow{\text{GH}} \mathcal{X}_{\infty}$ in the sense of Gromov-Hausdorff. 
\end{definition}

\begin{definition}
    Let $\mathcal{X}$ and $\mathcal{Y}$ be metric spaces. A map $f \colon X \to Y$ is called an $\varepsilon$-isometry if the following two conditions hold:
    \begin{enumerate}
        \item $\mathcal{X}$ is an $\varepsilon$-net in $\mathcal{Y}$; that is, for any $y \in \mathcal{Y}$ there is $x \in \mathcal{X}$ such that $|f(x)-y|_{\mathcal{Y}} < \varepsilon$.
        \item $\big||f(x) -f(x')|_{\mathcal{Y}} - |x-x'|_{\mathcal{X}} \big|\leq \varepsilon$ for any $x, x' \in \mathcal{X}$. 
    \end{enumerate}
    
\end{definition}

\begin{lemma}[\cite{Journey}]
    Let $\mathcal{X}_1, \mathcal{X}_2, ...$, and $\mathcal{X}_{\infty}$ be complete metric spaces, and $\varepsilon_n \to 0+$ as $n \to \infty$. Suppose that for each $n$ there is an $\varepsilon_n$-isometry $f_n\colon \mathcal{X}_n \to \mathcal{X}_{\infty}$. Then there exists a Gromov-Hausdorff convergence $\mathcal{X}_n \xrightarrow{\text{GH}} \mathcal{X}_{\infty}$. 
\end{lemma}

\begin{remark}
    To prove Lemma 1 it is enough to construct a metric $d$ on a common space $\mathbf{X}$ and show that $\mathcal{X}_n \to \mathcal{X}$ in the sense of Hausdorff. The metric $d$ can be defined in the following way: 
$$|x_n - y_n|_{\mathbf{X}} = |x_n - y_n|_{\mathcal{X}_n}$$
$$|x_{\infty} - y_{\infty}|_{\mathbf{X}} = |x_{\infty} - y_{\infty}|_{\mathcal{X}_{\infty}}$$
$$|x_{n} - x_{\infty}|_{\mathbf{X}} = \inf_{y_n \in \mathcal{X}_{n}}\{|x_{n} - y_n|_{\mathcal{X}_{n}} + \varepsilon_n + |f_n(y_n) - x_{\infty}|_{\mathcal{X}_{\infty}} \}$$
$$|x_{n} - x_{m}|_{\mathbf{X}} = \inf_{y_{\infty} \in \mathcal{X}_{\infty}}\{|x_{n} - y_{\infty}|_{\mathbf{X}} + |x_{m} - y_{\infty}|_{\mathbf{X}}\}$$
where $x_{\infty}, y_{\infty} \in \mathcal{X}_{\infty}$, and $x_{n}, y_{n} \in \mathcal{X}_{n}$ for each $n$.

\end{remark}

\begin{theorem}[Petrunin, \cite{Petrunin1}]

Assume $\mathcal{X}_n \xrightarrow{\text{GH}} \mathcal{X}_{\infty}$ without collapse (that is dim $\mathcal{X}_n \geq$ dim $\mathcal{X}_{\infty}$). And $\gamma_n \in \mathcal{X}_n$ is a sequence of quasigeodesics which in the $d$-metric (on $\mathbf{X}$) converges pointwise to $\gamma \in \mathcal{X}_{\infty}$. Then $\gamma$ is a quasigeodesic. 

\end{theorem}

Now we are ready to prove the following:

\begin{proposition}\label{th:Hopf}
        Let $M$ be a closed convex $n$-dimensional surface, and let $f\colon M \to \mathbb{R}^n$ be a continuous map. Then for any prescribed $\delta > 0$, there exists a pair $\{x,y\} \in M \times M$ such that $f(x) = f(y)$ and $x$ and $y$ are joined by a quasigeodesic of length~$\delta$.
\end{proposition}


\begin{proof}
    Let $B$ be the convex body bounded by $M$. Let $\{M_n\}_{n=1}^{\infty}$ be a sequence of smooth closed convex $n$-dimensional surfaces in $B$ such that $M_n$ converges to $M$ as $n \to \infty$ in the sense of Hausdorff.

    Let $p$ be a point inside $B$. And let $R_n$ be the radial projection with center at $p$ of $M_n$ on $M$. It is easy to see that each $R_n$ is $\varepsilon_n$-isometry for some $\varepsilon_n > 0$  and $\varepsilon_n \to 0$ as $n \to \infty$. Hence the sequence $\{M_n\}_{n=1}^{\infty}$ converges to $M$ in the sense of Gromov-Hausdorff by Lemma 1.

    For each $n > 0$ define a map $f_n = f \circ R_n: M_n \to \mathbb{R}^{n}$. By the Hopf theorem there exists a pair $\{x_n,y_n\} \in M_n \times M_n$ such that $f(x_n) = f(y_n)$ and $x_n$ and $y_n$ are joined by a geodesic of length~$\delta$ in $M_n$. Denote this geodesic parametrized by the arc length as $\gamma_n$. Passing to a subsequence in $\{\gamma_n\}_{n=1}^{\infty}$ if necessary we can assume by Arzela-Ascoli theorem that $\gamma_n$ converges to some continuous curve $\gamma$ of length $\delta$. Putting the metric $d$ from Remark 1 on the common space $\mathbf{X} = M \sqcup M_1 \sqcup M_2 \sqcup ...$, it follows that $\gamma_n$ converges pointwise to $\gamma$ as $n \to \infty$ in the metric $d$. Hence by Petrunin theorem $\gamma$ is quasigeodesic. This completes the proof.
    
\end{proof}

\section{Quantitative version of the Hopf theorem for convex surfaces}

\begin{definition}
	Let $n$ be a positive integer, and let $\delta$ be a positive real number. Let $M$ be a closed convex $n$-dimensional surface, and let $f\colon M \to \mathbb{R}^n$ be a continuous map. We denote by $\mathcal{F}(\delta)$ the subset of $M \times M$ such that $\{a,b\} \in \mathcal{F}(\delta)$ if and only if $f(a) = f(b)$ and the points $a$ and $b$ are joined by a quasigeodesic of length~$\delta$. 
		
    We call points $a, b \in M$ $\delta$-\emph{conjugate} if they are joined by an infinite number of quasigeodesics of length $\delta$. We denote the set of such points by $\mathcal{C}(\delta)$. 
    
\end{definition}
	
\begin{theorem}
    Let $n$ be a positive integer such that $n > 1$, and let $\delta$ be a positive real number. Let $M$ be a closed convex $n$-dimensional surface, and let $f\colon M \to \mathbb{R}^n$ be a continuous map. If $\mathcal{C}(\delta)$ is empty, then $\mathcal{F}(\delta)$ is uncountable. 
\end{theorem}

\begin{proof}
    The proof closely follows the corresponding argument for the same statement for Riemannian manifolds (proof of Theorem 5, Section 2 \cite{Previous}). 

    Let $B$ be the convex body bounded by $M$. Let $\{M_n\}_{n=1}^{\infty}$ be a sequence of smooth closed convex $n$-dimensional surfaces in $B$ such that $M_n$ converges to $M$ as $n \to \infty$ in the sense of Hausdorff. Let $p$ be a point inside $B$. And let $R_n$ be the radial projection with center at $p$ of $M_n$ on $M$.

   If $f(M)$ is a singleton, then the result easily follows \footnote{There is an infinite quasigeodesic starting in any direction (from a tangent cone) from any point \cite{Petrunin} of $M$.}. Otherwise $\partial{f(M)}$ is uncountable (since $f(M)$ is bounded and connected). 

    If no $\xi$ in $\partial f(M)$ yields $(f^{-1}(\xi) \times f^{-1}(\xi)) \cap F(\delta) = \emptyset$, then the statement easily follows. Otherwise, take $\xi \in \partial f(M)$ such that $(f^{-1}(\xi) \times f^{-1}(\xi)) \cap F(\delta) = \emptyset$ and let $q$ be a point in $f^{-1}(\xi)$. Let $q_n = R_n^{-1}(q)$. 

    Passing to a subsequence in $\{M_n\}_{n=1}^{\infty}$ if necessary, we can assume that there exists $r > 0$ such that the image of $T_{0,q_n,\delta}^{n}$ (see the proof of Theorem 5 in \cite{Previous} for the definition of $T_{s,q,\delta}$ \footnote{Let $p \in M_n$, and let $T_pM_n$ be the tangent space at $p$. Choose a standard basis in $T_pM_n$, and let $\mathbb{S}_{p}^{n-1} \subset T_pM_n$ be the unit sphere. For each $\mathfrak{F} \in \mathbb{S}_{p}^{n-1}$, there is a unique geodesic with the tangent vector $\mathfrak{F}$ at $p$, which we denote by $\gamma_{\mathfrak{F}}$. Study the vector $V(\mathfrak{F})_{0, p, \delta} := f(\gamma_{\mathfrak{F}}(\delta)) - f(p)$ and let  $v(\mathfrak{F})_{0, p, \delta} = \frac{V(\mathfrak{F})_{0, p, \delta}}{|V(\mathfrak{F})_{0, p, \delta}|}$, where $|V(\mathfrak{F})_{0, p, \delta}|$ denotes the Euclidean norm of $V(\mathfrak{F})_{0, p, \delta}$ in $\mathbb{R}^n$. Since $|V(\mathfrak{F})_{0, p, \delta}|$ is not zero by assumption, we get continuous maps $T_{0,p,\delta}^{n} \colon \mathbb{S}_{p}^{n-1} \to \mathbb{S}^{n-1}$.}) does not intersect a closed $n$-dimensional ball of radius $r$ with the center at $f(q)$. Note also that the degree of $T_{0,q_n,\delta}^{n}$ is zero \cite{Previous}, and hence from continuity of $f_n$ it follows that there exists $\varepsilon_{n} > 0$ such that for each $p \in U_{\varepsilon_{n}}(q_n)$, the degree of ${T}_{0,p,\delta}^{n}$ is zero as well, where $U_{\varepsilon_n}(q_n)$ is the $\varepsilon_n$-neighborhood of $q_n$ in $M_n$. We choose $\varepsilon_{n}$ as the maximal possible values for which the condition on degree of ${T}_{0,p,\delta}^{n}$ is satisfied. 

    From continuity of $f$ it follows that the sequence $\{\varepsilon_n\}_{n=1}^{\infty}$ is bounded from below. Hence the intersection of $R_n(U_{\varepsilon_{n}}(q_n))$ has non-empty interior and contains an open ball $\mathcal{B}$ of radius $r_1 > 0$. Let $U_n(q_n) = R_n^{-1}(\mathcal{B})$. Since through every point inside $U_n(q_n)$ passes a geodesic $\gamma^n_p$ containing $f_n$-neighbors with distance $\delta$ between them along $\gamma^n_p$ \cite{Previous}, by applying a similar argument as in the proof of Proposition 1 we get that through each point inside $\mathcal{B}$ passes a quasigeodesic $\gamma_p$ containing points in some pair in $\mathcal{F}(\delta)$ (the distance between points in this pair along $\gamma_p$ is $\delta$). 
    
    The rest of the proof is similar to \cite{Previous}: observe that no at most countable collection of quasigeodesics in $M$ covers $\mathcal{B}$, since any quasigeodesic has zero $n$-dimensional Lebesgue measure. Now if $\mathcal{F}(\delta)$ is at most countable, then there is a pair of points $\{\mu_1, \mu_2\} \in \mathcal{F}(\delta)$ such that $\mu_1$ and $\mu_2$ are joined by uncountably many quasigeodesics of length $\delta$ and hence $\mathcal{C}(\delta)$ is not empty. This completes the proof. \qedhere

\end{proof}
There is an interesting future direction of generalization of Hopf's original argument to the case of $n$-dimensional Alexandrov spaces $M$ of curvature $\ge \kappa$ without boundary. Is it possible to construct a continuous map, an analogue of exponential map, from an $n$-dimensional open ball to the Alexandrov space, such that the center of a ball is mapped to $p \in M$ and radial curves are mapped to quasigeodesics, starting from $p$ (different radial curves could be mapped to the same quasigeodesics)?
\section{Discrete version of the Hopf theorem}

Let $P$ be a triangulation \footnote{Recall that triangulation of topological space is simplicial complex homeomorphic to it. We take as a homeomorphism the identity map.} of the boundary of a convex polyhedron in $\mathbb{R}^3$. And let $f\colon P \to \mathbb{R}^2$ be a continuous map which is affine on each triangle of $P$ \footnote{$f$ can be given at vertices of $P$ and extended to $P$ by affine extension.}. Next, we assume that $f$ is a map of general position \footnote{No three vertex images lie on the same line.}. We recall that two points $a$ and $b$ of $P$ are called $f$-neighbors if $f(a) = f(b)$ \footnote{For convenience, we do not require that $a$ and $b$ are distinct.}. If $d$ is a metric on $P$, we say that $f$-neighbors $a$ and $b$ \textit{realize} distance $\delta \geq 0$ if $d(a,b) = \delta$. For $\delta > 0$ we define $\mathfrak{D}(\delta)$ to be the set of $f$-neighbors realizing distance $\delta > 0$. 

First we define the `induced by $f$' triangulation of $P$ by the following algorithm.

\begin{algorithm}
\caption{`Induced by $f$' triangulation of $P$}
\begin{algorithmic}[1]
    \State Denote by $T$ the triangulation of $f(P)$, which is constructed by the following three steps. 

    \begin{enumerate}
        \item Define \textit{vertex set of $T$} as intersection points of images of distinct edges of $P$. 
        \item Define \textit{constraints} $C$ as the set of line segments which arise from subdivision (induced by vertex set of $T$) of image of each edge of $P$. 
        \item Create constrained Delaunay triangulation $CD$ of $f(P)$ with constraints $C$. Triangulation $T$ of $f(P)$ is defined as the family of those simplices of $CD$ that belong to $f(P)$. 
    \end{enumerate}

    \State Construct `induced by $f$' triangulation of $P$ as the family of triangles (and their edges and vertices), which arise by the following two steps: 

    \begin{enumerate}
        \item Choose triangle $t$ of $T$.
        \item Take the closure of preimage of relative interior of $t$.
    \end{enumerate}

\end{algorithmic}
\end{algorithm}

\begin{figure}[H]
    \centering
    \begin{subfigure}[b]{0.25\textwidth}
        \centering
        \includegraphics[width=\linewidth]{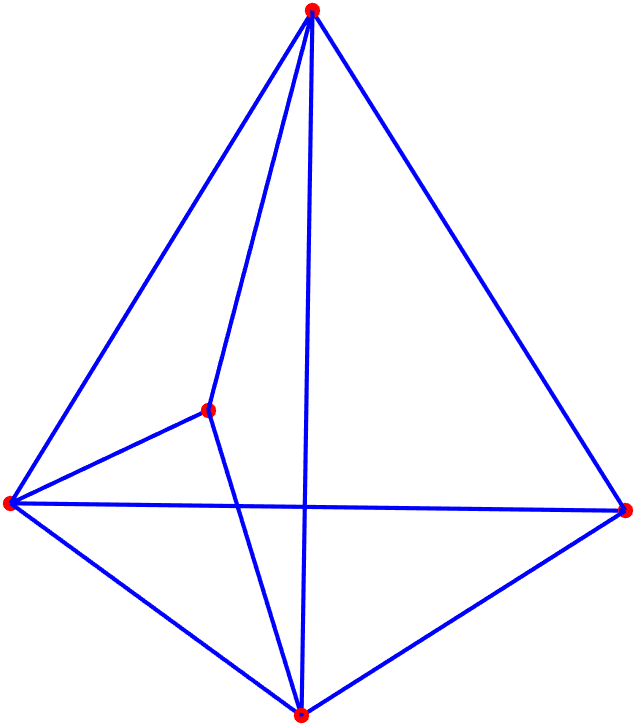}
        \label{fig:first}
    \end{subfigure}
    \hspace{1cm} 
    \begin{subfigure}[b]{0.25\textwidth}
        \centering
        \includegraphics[width=\linewidth]{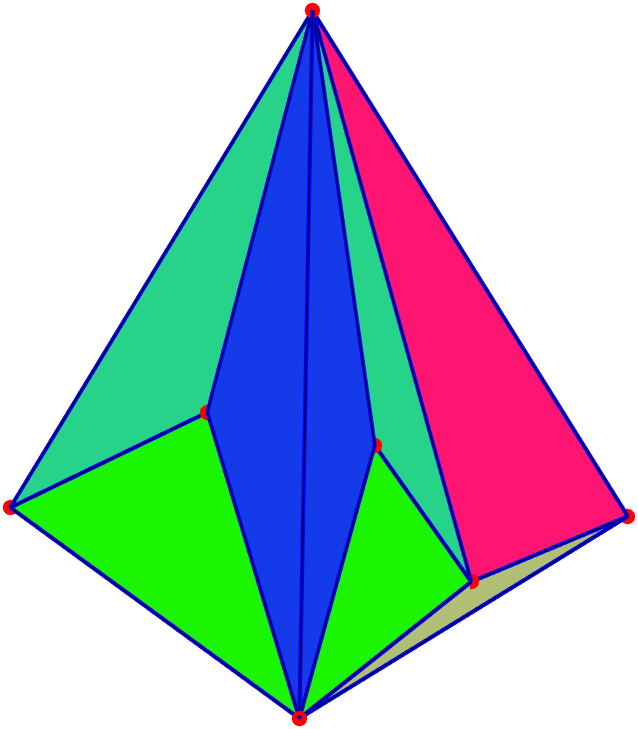}
        \label{fig:second}
    \end{subfigure}
    \caption{Initial and `induced by $f$' triangulations of $P$. Triangles, whose images coincide, are colored identically.}
    \label{fig:side_by_side}
\end{figure}

We denote the `induced by $f$' triangulation of $P$ as $T_{f}P$. An example of such triangulation is shown in Fig. 1. To state our main result of this section we need several auxiliary definitions, constructions and results.

For a given map $f$ we introduce the notion of \textit{a complex of $f$-neighbors}. 

\begin{definition}[Complex of $f$-neighbors]

For any two distinct triangles $\mathcal{A}$ and $\mathcal{B}$ of $T_{f}P$, whose images coincide, we define two triangles of corresponding $f$-neighbors in $P \times P \subset \mathbb{R}^3 \times \mathbb{R}^3$, which we denote by $[\mathcal{A}, \mathcal{B}]$ and $[\mathcal{B}, \mathcal{A}]$ (we will denote corresponding edges in the same way). We construct complex of $f$-neighbors $\mathfrak{F}_f$ as a family of such triangles, their edges and vertices. Notice that $\mathfrak{F}_f$ coincides with the closure (in $P \times P$) of the set of $f$-neighbors realizing non-zero distances.  
    
\end{definition}


\begin{lemma}
    Let $P$ be a triangulation of the boundary of a convex polyhedron in Euclidean space $\mathbb{R}^3$. And let $f\colon P \to \mathbb{R}^2$ be a continuous map which is affine on each triangle of $P$ and of general position. Then the associated complex of $f$-neighbors $\mathfrak{F}_f$ is a simplicial complex and any edge of $\mathfrak{F}_f$ belongs to exactly two triangles of $\mathfrak{F}_f$ \footnote{The MATLAB implementation of construction of $T_fP$ and $\mathfrak{F}_f$ is available upon request.}. 
\end{lemma}

\begin{proof}
    The first part of the statement follows from the construction of $\mathfrak{F}_f$. Next we focus on the second part. Let $[\mathcal{A}, \mathcal{B}]$ be a triangle in $\mathfrak{F}_f$ and let $l$ be an edge of $[\mathcal{A}, \mathcal{B}]$. There are the edges $l_{\mathcal{A}}$ and $l_{\mathcal{B}}$ in triangles $\mathcal{A}$ and $\mathcal{B}$ respectively such that $l = [l_{\mathcal{A}}, l_{\mathcal{B}}]$. We have two cases:
    \begin{enumerate}
        \item If $l_{\mathcal{A}}$ coincide with $l_{\mathcal{B}}$, then triangles $[\mathcal{A}, \mathcal{B}]$ and $[\mathcal{B}, \mathcal{A}]$ share the edge $l$. 
        \item If $l_{\mathcal{A}}$ and  $l_{\mathcal{B}}$ are distinct, denote by $\mathcal{A}_1$ and $\mathcal{B}_1$ triangles adjacent to edges $l_{\mathcal{A}}$ and  $l_{\mathcal{B}}$ respectively besides $\mathcal{A}$ and $\mathcal{B}$. It follows that triangles $\mathcal{A}, \mathcal{A}_1, \mathcal{B}, \mathcal{B}_1$ are all distinct. Since $f$ is of general position we have two subcases:

        \begin{enumerate}[label=(\alph*)]
            \item `no foldings': $f(\mathcal{A}_1)$ coincides with $f(\mathcal{B}_1)$, but not with $f(\mathcal{A})$ and $f(\mathcal{B})$. In this subcase triangles $[\mathcal{A}, \mathcal{B}]$ and $[\mathcal{A}_1, \mathcal{B}_1]$ share the edge $l$. 
            \item `one folding': either $f(\mathcal{A}_1)$ or $f(\mathcal{B}_1)$ (but not both) coincides with $f(\mathcal{A})$ and $f(\mathcal{B})$. In this subcase either triangles $[\mathcal{A}, \mathcal{B}]$ and $[\mathcal{A}_1, \mathcal{B}]$, or either triangles $[\mathcal{A}, \mathcal{B}]$ and $[\mathcal{A}, \mathcal{B}_1]$, share the edge $l$.
        \end{enumerate}
    \end{enumerate}

    We observe that in each case only the specified pair of triangles share the edge $l$. 
    
\end{proof}

\begin{lemma}
    Let $n$ be a positive integer. Let $M$ be a closed manifold of dimension $n$ with non-trivial $\mathbb{Z}_2$ action. And let $\mathbb{S}^n$ be a Euclidean unit $n$-sphere in the Euclidean $(n + 1)$-space $\mathbb{R}^{n+1}$ with standard antipodal $\mathbb{Z}_2$ action. Let $f \colon M \to \mathbb{S}^n$ be a continuous map of an odd degree. Then $f$ is $\mathbb{Z}_2$-equivariant at least at one point of $M$. 
\end{lemma}

\begin{proof}
    Suppose to the contrary that $f$ is not $\mathbb{Z}_2$-equivariant at any point of $M$. Then we can homotope $f$ to $f'$ moving points $f(x)$ and $f(\mathbb{Z}_2(x))$ to the midpoint of the shortest arc connecting them for each $x \in M$. By simplcial approximation theorem it follows that
    there exits a sufficiently small triangulation $T(M)$ of $M$ such that $f'$ is homotopic to a continuous map $f''$, which agrees with $f'$ at vertices of $T(M)$ and affinely extended to each simplex of $T(M)$. By a further subdivision of $T(M)$ we can assume that the set of vertices of $T(M)$ is $\mathbb{Z}_2$ invariant. The map $f''$ has an even degree by construction. This a contradiction.

\end{proof}

\begin{definition}[Antipodal points for closed convex surface]

Let $B$ be an $(n+1)$-dimensional convex body in $\mathbb{R}^{n+1}$ and $P$ be its boundary. And let $p$ be a point in $\Int B$ \footnote{We denote by $\Int X$ the interior of a topological space $X$.}. There is natural $\mathbb{Z}_2$ action on $P$ associated with the point $p$: if $l$ is a line through $p$, then its intersection points with $P$ are reversed by the action. Let $d$ be a metric on $P$. We set

$$\mathcal{D}(P, d) = \sup_{p \in \Int B}{\inf_{x \in P} {d(x, \mathbb{Z}_{2}(x))}}$$

The point $p$, for which supremim above is achieved, we call the center of $P$. Points $a$ and $b$ are called \textit{antipodal} if $a = \mathbb{Z}_{2}(b)$. By Borsuk-Ulam theorem for any continuous map $f\colon P \to \mathbb{R}^n$ there are always exist antipodal points $a$ and $b$ such that $f(a) = f(b)$ and $d(a,b) \geq \mathcal{D}(P,d)$. 
\end{definition}

\begin{theorem}[Discrete Hopf]
    Let $(P,d)$ be a triangulation of the boundary of a convex polyhedron in Euclidean space $\mathbb{R}^3$, with metric $d$, compatable with topology of $P$. And let $f\colon P \to \mathbb{R}^2$ be a continuous map which is affine on each triangle of $P$ and of general position. Then there exists a polygonal path of $f$-neighbors in $\mathfrak{F}_f$ that connects a pair of antipodal points with a pair of identical points. 
\end{theorem}

\begin{proof}

    Observe that $\mathfrak{F}_f$ can have singularities at vertices that lie on the diagonal of $P \times P$. First we construct another abstract simplicial complex $\mathfrak{L}_f$ from $\mathfrak{F}_f$ resolving singularities by the following algorithm. For each family of triangles of $\mathfrak{F}_f$ that form a cycle around a vertex in $\mathfrak{F}_f$ we define a corresponding vertex in $\mathfrak{L}_f$. We define triangles and edges $\mathfrak{L}_f$ as sets of vertices of $\mathfrak{L}_f$ such that corresponding to them vertices of $\mathfrak{F}_f$ form an edge or a triangle in $\mathfrak{F}_f$. 
    
    Next we identify $\mathfrak{L}_f$ with some of its geometric realization. Note that there is natural simplicial map $i \colon \mathfrak{L}_f \to \mathfrak{F}_f$ which possibly glues some vertices $\mathfrak{L}_f$ together. 
    
    Let $pr \colon P \times P \to P$ be a projection. Denote by $j$ the composition of $i$ and $pr$. Observe that there exists a component $\mathcal{L}$ of $\mathfrak{L}_f$ such that $j \colon \mathcal{L} \to P$ has degree $1$ and $i(\mathcal{L})$ intersects the diagonal of $P \times P$. Indeed since $f$ is of general position there is a pair of adjacent triangles $\mathcal{A}$ and $\mathcal{B}$ in $P$ such that $f(\mathcal{A}) = f(\mathcal{B})$ (`folding') and there are no other triangles in $P$ whose images coincide with $f(\mathcal{A})$ and $f(\mathcal{B})$. Thus we can take $\mathcal{L}$ as a component of $\mathcal{L}_f$ that contains $i^{-1}([\mathcal{A}, \mathcal{B}])$. 

    There is natural $\mathbb{Z}_2$ action on $\mathfrak{F}_f$ which is the reflection about the diagonal of $P\times P$. This action naturally lifts to $\mathcal{L}$. By Lemma 3 there exists 
    a point $x \in \mathcal{L}$ such that $j$ is $\mathbb{Z}_2$-equivariant at $x$ ($\mathbb{Z}_2$ action on $P$ is choosed relative to the center of $P$). Hence the same is true for map $Pr$ and the point $i(x)$. That is there exists a pair $(p,q)$ of antipodal $f$-neighbors in $\mathcal{L}$. Taking any polygonal path from $(p,q)$ to the diagonal completes the proof.  
    
\end{proof}

\begin{remark}[``Starship Enterprise'']
    Observe that a component $\mathcal{L}$ could be topologically non-trivial. It is an interesting exercise to see that in the case of a map $f\colon S^2 \to \mathbb{R}^2$ shown schematically on Fig. 2, $\mathfrak{L}_f$ is connected and homeomorphic to a torus (hence the same is true for $\mathcal{L}$). It's easy to generalize this example so that $\mathcal{L}$ can have any non-negative genus.

    \begin{figure}[h!]
        \centering
        \includegraphics[width=0.5\textwidth]{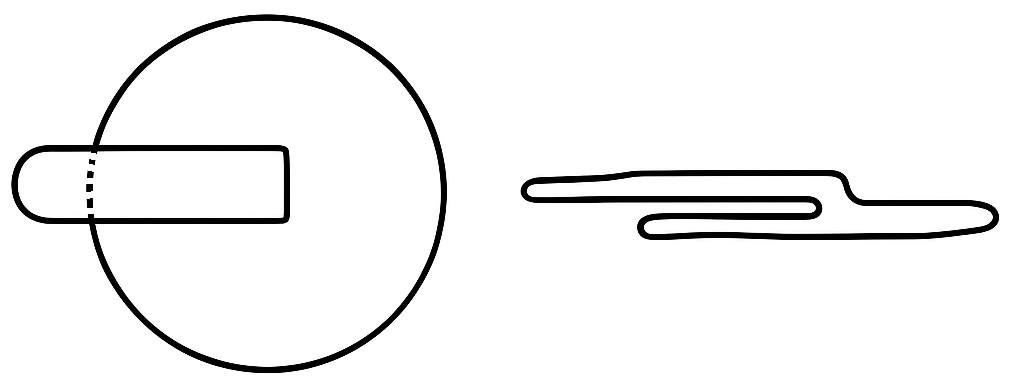}
        \caption{Schematic picture of a map $f\colon S^2 \to \mathbb{R}^2$. Top view on the left, front view on the right (before projection on a plane).}
        \label{fig:Lf_diagram}
    \end{figure}
\end{remark}

To explore the structure of of the set $\mathfrak{D}(\delta)$ of $f$-neighbors realizing distance $\delta > 0$ we need the following auxiliary statement, which is apparently well known in the literature, but we couldn't find the exact formulation. 

\begin{lemma}
    Let $n$ be a positive integer. And let $M$ be a connected closed orientable $n$-dimensional manifold. And let $K$ be a compact subspace of $M$ such that $M \setminus K$ is  disconnected. Then $(n-1)$-dimensional Steenrod homology $H_{n-1}(K; \mathbb{Z})$ are not trivial.
\end{lemma}

\begin{proof}
    By Theorem 11.9 of \cite{Massey} the cohomological sequence of the pair $(M, M \setminus K)$ is isomorphic to the homological sequence of the pair $(M, K)$. Thus we have the following commutative diagram, where the rows are exact and vertical arrow are isomorphisms: 

    \[
    \begin{tikzcd}
    H^{0}(M; \mathbb{Z}) \arrow[d] \arrow[r] & H^{0}(M\setminus K; \mathbb{Z}) \arrow[d] \arrow[r] & H^1(M, M \setminus K; \mathbb{Z}) \arrow[d] \\
    H_n(M; \mathbb{Z}) \arrow[r] & H_n(M,K; \mathbb{Z}) \arrow[r] & H_{n-1}(K; \mathbb{Z})
    \end{tikzcd}
    \] \footnote{Here $H^q$ denotes the Alexander-Spanier cohomology, defined in  \cite{Massey} for general spaces. For paracompact Hausdorff spaces Alexander-Spanier cohomology coincides with the Alexandrov-\v{C}ech cohomology. See also \cite{Steenrode} for the constructive definition of Steenrod homology.} 

    By Theorem 4.20 in \cite{Massey} $H_n(M; \mathbb{Z}) \cong \mathbb{Z}$. Since $H^{0}(M\setminus K; G) \cong \mathbb{Z}^{r}$, where $r$ denotes the number of connected components of $M\setminus K$ (which can be infinite), we have the following exact sequence: $\mathbb{Z} \to \mathbb{Z}^r \to H_{n-1}(K; \mathbb{Z})$. Since $r > 1$, the statement is proved.

\end{proof}

\begin{proposition}
    Under assumptions of Theorem 4, 
    the set $\mathfrak{D}(\delta)$ of $f$-neighbors realizing a given distance $\delta \in (0, \mathcal{D}(P,d))$, has non-trivial first Steenrod homology with coefficients in $\mathbb{Z}$. 
\end{proposition}

\begin{proof}
    The distance function $d$ naturally lifts from complex of $f$-neighbors $\mathfrak{F}_f$ to the continuous function $\tilde{d}$ on $\mathcal{L}_f$ (see the proof of the Theorem 4 for notation). By Theorem 4, the exists a component $\mathcal{L}$ of $\mathcal{L}_f$ such that $\tilde{d}(\mathcal{L}) = [0, \mathcal{D}(P,d)]$. Since the projection $i \colon \mathcal{L} \to \mathfrak{F}_f$ is a homeomorphism except possibly the vertices, where $\tilde{d} = 0$, the set $C_{\delta} = \{x \in \mathcal{L} ~ | ~ \tilde{d}(x) = \delta\}$ is homeomorphic to $\mathfrak{D}(\delta)$. Applying Lemma 3 to $(\mathcal{L}, C_{\delta})$ completes the proof.

\end{proof}

\section{Acknowledgments}

I am grateful to Oleg R. Musin for pointing out the research direction of generalization of the Hopf theorem to the case of convex surfaces, interesting discussions and inspiration. 

I sincerely thank Andrey Grigoryev bringing to my attention several valuable papers of A.D.Alexandrov found in the laboratory office.

\end{document}